\documentclass[11pt]{article}
\usepackage{amsfonts,latexsym,amsmath,amssymb,amsthm}
\usepackage{algorithm}
\usepackage{xcolor}
\usepackage{dsfont}
\usepackage{abstract}
\usepackage{graphicx}
\usepackage[margin=1.1in]{geometry}
\usepackage{enumitem}
\definecolor{pred}{RGB}{148,55,61}
\usepackage[colorlinks,linkcolor=pred,citecolor=pred,anchorcolor=pred]{hyperref}

\newcommand{\ba}{\noindent $\begin{array}}
\newcommand{\ea}{\end{array}$}
\newcommand{\be}{\begin{equation}}
\newcommand{\ee}{\end{equation}}
\newcommand{\bd}{\begin{displaymath}}
\newcommand{\ed}{\end{displaymath}}
\newcommand{\beq}{\begin{eqnarray*}}
\newcommand{\eeq}{\end{eqnarray*}}
\newcommand{\beqn}{\begin{eqnarray}}
\newcommand{\eeqn}{\end{eqnarray}}
\newcommand*{\underuparrow}[1]{\ensuremath{\underset{{\uparrow}}{#1}}}
\DeclareMathOperator*{\argmin}{arg\,min}

\DeclareMathOperator*{\dom}{dom}

\def\hat{\widehat}
\def\tilde{\widetilde}
\def\[{\begin{equation}}
\def\]{\end{equation}}

\def\Diag{\textup{Diag}}
\def\norm#1{\|#1\|}
\def\inprod#1#2{\big\langle#1,\,#2\big\rangle}
\def\disp{\displaystyle}
\def\norm#1{\|#1 \|}
\def\inprod#1#2{\langle#1,\,#2 \rangle}
\def\cA{{\cal A}} \def\cB{{\cal B}} 
  \def\cP{{\cal P}} \def\cT{{\cal T}}
\def\cH{{\cal H}} \def\cU{{\cal U}} \def\cG{{\cal G}} 
 \def\cM{{\cal M}} \def\cN{{\cal N}} \def\cW{{\cal W}}
\def\cS{{\cal S}} \def\cE{{\cal E}} \def\cF{{\cal F}} \def\cV{{\cal V}}
 \def\cX{{\cal X}} \def\cY{{\cal Y}} \def\cZ{{\cal Z}}
\def\cN{{\cal N}} \def\cL{{\cal L}} \def\cS{{\cal S}}

\numberwithin{equation}{section}
\newtheorem{theorem}{Theorem}[section]
\newtheorem{proposition}{Proposition}[section]
\newtheorem{lemma}{Lemma}[section]

\newtheorem{remark}{Remark}[section]

\newtheorem{assumption}{Assumption}[section]

\title{\Large\bf A Unified Algorithmic Framework of Symmetric Gauss-Seidel Decomposition based Proximal ADMMs for Convex Composite Programming}
\date{\today}
\author{Liang Chen\thanks{College of Mathematics and Econometrics, Hunan University, Changsha, 410082, China ({\tt chl@hnu.edu.cn}), and Department of Applied Mathematics, The Hong Kong Polytechnic University, Hong Kong, China ({\tt liangchen@polyu.edu.hk}). This author was supported by the National Natural Science Foundation of China (11801158, 11871205), the Hunan Provincial Natural Science Foundation of China (2019JJ50040), and the Fundamental Research Funds for Central Universities in China.}
\and Defeng Sun\thanks{Department of Applied Mathematics, The Hong Kong Polytechnic University, Hung Hom, Hong Kong ({\tt defeng.sun@polyu.edu.hk}). This author was supported by a start-up research grant from the Hong Kong Polytechnic University.}
\and Kim-Chuan Toh\thanks{Department of
Mathematics, and Institute of Operations Research and Analytics, National University of Singapore, 10 Lower Kent Ridge Road, Singapore 119076 ({\tt mattohkc@nus.edu.sg}). This author was supported in part by the Ministry of Education, Singapore, Academic Research Fund (R-146-000-256-114).}
\and Ning Zhang\thanks{College of Computer Science and Technology, Dongguan University of Technology, Dongguan, China ({\tt ningzhang\_2008@yeah.net}), and Department of Applied Mathematics, The Hong Kong Polytechnic University, Hung Hom, Hong Kong, China (\tt ning.nz.zhang@polyu.edu.hk).}}
\begin{document}
\maketitle

\begin{abstract}
This paper aims to present a fairly accessible generalization of several symmetric Gauss-Seidel decomposition based multi-block proximal alternating direction methods of
multipliers (ADMMs) for convex composite optimization problems.
The proposed method unifies and refines many constructive techniques that were separately developed for the computational efficiency of multi-block ADMM-type algorithms.
Specifically, the majorized augmented Lagrangian functions, the indefinite proximal terms, the inexact symmetric Gauss-Seidel decomposition theorem, the tolerance criteria of approximately solving
the subproblems, and the large dual step-lengths, are all incorporated in one algorithmic framework, which we named as {\sf sGS-imiPADMM}.
From the popularity of convergent variants of multi-block ADMMs in recent years, especially for high-dimensional multi-block convex composite conic programming problems, the unification presented
in this paper, as well as the corresponding convergence results, may have the great potential of facilitating the implementation of many multi-block ADMMs in various problem settings.
\end{abstract}
\smallskip
{\small
\begin{center}
\parbox{0.95\hsize}{{\bf Keywords.}\;
Convex optimization,
multi-block,
alternating direction method of multipliers,
symmetric Gauss-Seidel decomposition,
majorization.
Majorization}
\end{center}
\begin{center}
\parbox{0.95\hsize}{{\bf AMS Subject Classification.}\; 90C25, 90C22, 90C06, 65K05.}
\end{center}}

\section{Introduction}
In this paper, we consider the following multi-block convex composite programming:
\[
\label{problem:primal}
\min_{x\in\cX,\,y\in\cY} \big\{
 p_{1}(x_1)+f(x_1,\ldots,x_{m})+q_{1}(y_1)+g(y_{1},\ldots,y_{n})\,
|\, \cA^* x+\cB^* y=c
\big\},
\]
where $\cX$, $\cY$ and $\cZ$ are three finite dimensional real Hilbert spaces, each endowed with an inner product $\langle\cdot,\cdot\rangle$ and its induced norm $\|\cdot\|$, and
\begin{itemize}
[topsep=1pt,itemsep=-.6ex,partopsep=1ex,parsep=1ex,leftmargin=4ex]
\item[-]
$\cX$ can be decomposed as the Cartesian product of $\cX_{1},\ldots, \cX_{m}$, which are finite dimensional real Hilbert spaces endowed with the inner product $\langle \cdot,\cdot\rangle$ inherited from $\cX$ and its induced norm $\|\cdot\|$. Similarly, $\cY=\cY_1\times\cdots\times\cY_n$. Based on such decompositions, one can write $x\in\cX$ as $x=(x_1,\ldots,x_m)$ with $x_i\in\cX_i,\ i=1,\ldots,m$, and, similarly, $y=(y_1,\ldots,y_n)$;
\item[-]
$p_{1}:\cX_1\to(-\infty,\infty]$ and $q_{1}:\cY_{1}\to(-\infty,\infty]$ are two closed proper convex functions;

\item[-]
$f:\cX\to(-\infty,\infty)$ and $g:\cY\to(-\infty,\infty)$ are continuously differentiable convex functions with Lipschitz continuous gradients;

\item[-]
$\cA^*$ and $\cB^*$ are the adjoints of the two given linear mappings $\cA:\cZ\to\cX$ and $\cB:\cZ\rightarrow \cY$, respectively; $c\in\mathcal{Z}$ is a given vector;

\item[-]
without {loss} of generality, we define the two functions $p:\cX\to(-\infty,\infty]$ and $q:\cY\to(-\infty,\infty]$ by $p(x):=p_{1}(x_{1}),\ \forall x\in\cX$ and $q(y):=q_{1}(y_{1}),\ \forall y\in\cY$ for convenience.

\end{itemize}

\medskip
At the first glance, one may view problem \eqref{problem:primal} as a $2$-block separable convex optimization problem with coupled linear equality constraints. Consequently, the classic alternating direction method of multipliers (ADMM) \cite{glowinski1975approximation,Gabay1976} and its contemporary variants such as \cite{eckstein1994,fazel2013hankel} can be used for solving problem \eqref{problem:primal}.
For the classic $2$-block ADMM, one may refer to \cite{Eckstein2015,Glowinski2014} for a history of the algorithm and to the recent note \cite{chennote} for a thorough study on its convergence properties.

In high-dimensional settings, it is usually not computationally economical to directly apply the $2$-block ADMM and its variants to solve problem \eqref{problem:primal}, as in this
case solving the subproblems at each ADMM iteration can be too expensive. The difficulty is
made more severe especially when we know that ADMMs, being
intrinsically first-order methods, are prone to require
 a large number of outer iterations to
 compute even a moderately accurate approximate solution.
As a result, further decomposition of the variables in problem \eqref{problem:primal} for getting easier subproblems, if possible, should be incorporated when designing ADMM-type methods for solving it.
Unfortunately, even if the functions $f$ and $g$ in problem \eqref{problem:primal} are separable with respect to each subspace, say, $\cX_i$ and $\cY_j$, the naive extension of the classic ADMM to multi-block cases is not necessarily convergent\cite{CHYY2016}. How to address the aforementioned issues is the key reason why the algorithmic development, as well as the corresponding convergence analysis,
of multi-block variants of the ADMM has been an important research topic in convex optimization.

Of course, it is not reasonable to expect finding a general algorithmic framework that can achieve sterling numerical performance on a wide variety of different classes of linearly constrained multi-block convex optimization problems.
Thus, in this paper our focus is on model \eqref{problem:primal},
which is already quite versatile, for the following two reasons. Firstly, this model is general enough to handle quite a large number of convex composite optimization models from both the core
convex optimization and realistic applications
\cite{li2016schur,chen}. Secondly, the convergence of multi-block variants of the ADMM for solving problem \eqref{problem:primal} has been separately realized in
\cite{STY2015,li2016schur,LiMajADMM2016,chen,zhangn,chen18},
without sacrificing the numerical performance when compared to the naively extended multi-block ADMM. The latter has long been served as a benchmark for
comparing new ADMM-type methods since its impressive numerical
performance has been well recognized in extensive numerical experiments, despite its
lack of theoretical convergence guarantee.
Currently, this line of ADMMs has been applied to many concrete instances of problem \eqref{problem:primal}, e.g., \cite{BAIQI2016,DQi2016,Lam2018,Wang2018,yangaoteo,yst2015,Ferreira2018,Song2018,SDPNALPLUS2017,QSDPNAL}, to name just a few.

Motivated by the above exposition, in this paper, we plan to propose a unified multi-block ADMM for solving problem \eqref{problem:primal}. Our unified method is an {\it inexact symmetric Gauss-Seidel} (sGS) {\it decomposition based majorized indefinite-proximal ADMM} ({\sf sGS-imiPADMM}).
The purpose of this study is to distill and synthesize all the practical techniques that were constructively exploited in the references mentioned above for the computational efficiency of multi-block ADMM-type algorithms. Specifically, our unified algorithm incorporates all the following ingredients
developed over the past few years:
\begin{itemize}[topsep=1pt,itemsep=-.6ex,partopsep=1ex,parsep=1ex,leftmargin=4ex]
\item[-] the inexact sGS decomposition theorem progressively developed in \cite{STY2015,li2016schur,lisgs};
\item[-] the majorized augmented Lagrangian functions introduced in \cite{LiMajADMM2016};
\item[-] the indefinite proximal terms studied in \cite{LiMajADMM2016,zhangn,chen18};
\item[-] the admissible stopping conditions of approximately solving the subproblems
developed in \cite{chen},
\end{itemize}
together with the large dual step-lengths for the classic ADMM \cite{glowinski1975approximation}.

We show that the proposed {\sf sGS-imiPADMM} is globally convergent under very mild assumptions, and the resulting convergence theorem also improves those in the highly related references mentioned above.
For instance, it refines the {\sf sGS-imsPADMM} in \cite{chen} by substituting the extra condition \cite[(5.26) of Theorem 5.1]{chen} for establishing the convergence with the weaker basic condition\cite[(3.2)]{chen}, which is imposed for the well-definedness of the algorithm.
Moreover, compared with the recently developed sGS decomposition based majorized ADMM with indefinite proximal terms in \cite{zhangn}, the problem setting for {\sf sGS-imiPADMM} is much more general as the functions $f$ and $g$ here are not restricted to have the separable structures that were required in \cite{zhangn}.

The rest of this paper is organized as follows. In section \ref{sec:pre}, we
introduce some notation, and recall the inexact sGS decomposition theorem which plays an important role in the subsequent algorithmic design.
In section \ref{sec:sgs-imipadmm}, the {\sf sGS-imiPADMM} algorithm for the multi-block problem \eqref{problem:primal} is formally proposed.
Its global convergence theorem is established in section \ref{sec:proof} via the convergence of an {\it inexact majorized indefinite-proximal ADMM} ({\sf imiPADMM}), which will be introduced in section \ref{sec:imipadmm}, in conjunction with the inexact sGS decomposition theorem. Finally, we conclude the paper in section \ref{sec:conclude}.

\section{Notation and Preliminaries}\label{sec:pre}

\subsection{Notation}
Let $\cU$ and $\cV$ be two arbitrary finite dimensional real Hilbert spaces each endowed with an inner product $\langle\cdot,\cdot\rangle$ and its induced norm $\|\cdot\|$.
For any given linear map $\cH:\cU\to\cV$, we use $\|\cH\|$ to denote its spectral norm and $\cH^*:\cV\to\cU$ to denote its adjoint linear operator.

If $\cU=\cV$ and $\cH$ is self-adjoint and positive semidefinite, then there exists a unique
self-adjoint positive semidefinite linear operator $\cH^{\frac{1}{2}}:\cU\to\cU$ such that
$\cH^{\frac{1}{2}}\cH^{\frac{1}{2}}=\cH$.
In this case, for any $u,v\in\cU$, we define $\langle u,v\rangle_{\cH}:=\langle u,\cH v\rangle$ and $\|u\|_\cH:=\sqrt{\langle u, \cH u\rangle}=\|\cH^{\frac{1}{2}}u\|$.

For a closed proper convex function $\theta:\cU\to(-\infty,+\infty]$,
we denote $\dom\theta$ and $\partial\theta$ for the effective domain and the subdifferential mapping of $\theta$, respectively.

Let $s>0$ be a given integer such that one can decompose $\cU$ as the Cartesian product of $\cU_{1},\ldots, \cU_{s}$, such that each $\cU_i$ is a finite dimensional real Hilbert space endowed with the inner product $\langle \cdot,\cdot\rangle$ inherited from $\cU$ and its induced norm $\|\cdot\|$.
Then, the self-adjoint positive semidefinite linear operator $\cH:\cU\to\cU$ can be symbolically decomposed as
\[
\label{eq-Hu}
\cH =
\begin{pmatrix}
\cH_{11}&\cH_{12}&\cdots &\cH_{1s}\\[2mm]
\cH^*_{12}&\cH_{22}&\cdots &\cH_{2s}\\[2mm]
\vdots&\vdots&\ddots&\vdots \\[2mm]
\cH^*_{1s}&\cH^*_{2s}&\cdots &\cH_{ss}
\end{pmatrix},
\]
where $\cH_{ij}:\cU_{j}\to\cU_{i}$,
$i,j= 1,\ldots,s$ are linear maps and $\cH_{ii},\,i=1,\ldots,s$ are self-adjoint positive semidefinite linear operators. Based on \eqref{eq-Hu}, we use $\cH_d:=\Diag(\cH_{11},\ldots,\cH_{ss})$ to denote the block-diagonal part of $\cH$, and denote the symbolically {strictly} upper triangular part of $\cH$ by $\cH_u$, so that $\cH=\cH_{d}+\cH_{u}+\cH_{u}^{*}$. To simplify the notation in this case, for any $u=(u_1,\ldots,u_s)\in\cU$ and $i\in\{1,\ldots,s\}$, we denote
$u_{\le i}:=\{u_1,\ldots,u_i\}$, $u_{\ge i}:=\{u_i,\ldots,u_s\}$.

\subsection{The inexact sGS decomposition theorem}
We now briefly review the inexact block sGS decomposition theorem in \cite{lisgs}, which is a generalization of the Schur complement based decomposition technique developed in \cite{li2016schur}.

Following the notation of the previous subsection, suppose that $\cU=\cU_{1}\times\cdots\times\cU_{s}$ and $\cH$ is symbolically decomposed as in \eqref{eq-Hu}. Let $\theta_1:\cU_{1}\to(-\infty,\infty]$ be a given closed proper convex function, and $b\in \cU$ be a given vector.
Define the convex quadratic function $h:\cU\to(-\infty,\infty)$ by
$$
h(u):=\frac{1}{2}\langle u,\cH u\rangle-\langle b,u\rangle,\quad \forall u\in\cU.
$$
Let $\tilde{\delta}_i,\delta_i\in\cU_i$, $i=1,\ldots,s$ be given (error tolerance) vectors with
$\tilde{\delta}_1=\delta_{1}$. With the assumption that $\cH_{d}$ is positive definite, we define
\begin{eqnarray}
\label{def:d-general}
d(\tilde\delta,\delta)
:=\delta +
\cH_{u}\cH_{d}^{-1}
(\delta- \tilde{\delta})
\quad\mbox{with}
\quad
\delta:=(\delta_1,\ldots,\delta_s)\ \mbox{and}\ \tilde\delta:=(\tilde\delta_1,\ldots\tilde\delta_s).
\end{eqnarray}
Suppose that $u^{-}\in\cU$ is a given vector. Define
\[
\label{steps}
\left\{
\begin{array}{lll}
\disp
\widetilde u_i
 := \argmin_{u_i}
\big\{\theta(u^{-}_1)
+h(u^{-}_{\le i-1},u_i,\widetilde u_{\ge i+1})
-\langle\widetilde\delta_i,u_i\rangle
\big\},
&i=s, \dots,2,
\\[2mm]
\disp
u_1^+
 :=\argmin_{u_1}\big\{\theta(u_1)+h(u_1,\widetilde u_{\ge2})-\langle\delta_1,u_1\rangle\big\},
\\[2mm]
\disp
u_i^{+}
 :=\argmin_{u_i}
\big\{
\theta(u_1^+)
+h(u^+_{\le i-1},u_i,\widetilde u_{\ge i+1})-\langle\delta_i,u_i\rangle\big\},
&i=2,\ldots,s.
\end{array}
\right.
\]
Meanwhile, define the self-adjoint positive semidefinite linear operator ${\rm sGS}(\cH):\cU\to\cU$ by
\[
\label{sgsop}
{\rm sGS}(\cH) :=\cH_{u}\cH_{d}^{-1}\cH_{u}^{*}.
\]
Now, consider the following convex composite quadratic optimization problem:
\[
\label{wplus}
\min_{u\in\cU}\left\{
\theta(u_1)
+h(u)
+\frac{1}{2}\|u-u^{-}\|_{{\rm sGS}(\cH)}^2
-\inprod{d(\tilde\delta,\delta)}{u} \right\}.
\]
The following sGS decomposition theorem from \cite[Theorem 1 $\&$ Proposition 1]{lisgs},
reveals the equivalence between the sGS iteration \eqref{steps} and the proximal minimization problem \eqref{wplus}. This theorem is essential for the algorithmic development in this paper.
\begin{theorem}
\label{prop:sgs}
Suppose that $\cH_{d}=\Diag(\cH_{11},\ldots,\cH_{ss})$ is positive definite. Then
\begin{itemize}[topsep=0pt,itemsep=-.6ex,partopsep=1ex,parsep=1ex]
\item[{\rm (i)}]
both the iteration process in \eqref{steps} and
the linear operator ${\rm sGS}(\cH)$ defined by \eqref{sgsop} are well-defined;

\item[{\rm (ii)}] problem \eqref{wplus} is well-defined and admits a unique solution, which is exactly the vector $u^+$ generated by \eqref{steps};

\item[{\rm (iii)}] it holds that $\hat\cH:=\cH+{\rm sGS}(\cH)=(\cH_{d}+\cH_{u})\cH_{d}^{-1}(\cH_{d}+\cH_{u}^{*})\succ 0$;

\item[{\rm (iv)}] the vector $d(\tilde{\delta},\delta)$ defined by \eqref{def:d-general} satisfies
$$
\|\hat\cH^{-\frac{1}{2}}d(\tilde{\delta},\delta)\|
\le \|\cH_{d}^{-\frac{1}{2}}(\delta-\tilde\delta)\|
+\|\cH_{d}^{\frac{1}{2}}(\cH_{d}+\cH_{u})^{-1}\tilde\delta\|.
$$
\end{itemize}
\end{theorem}

\section{An Inexact sGS Decomposition Based Majorized Indefinite-Proximal ADMM}
\label{sec:sgs-imipadmm}
In this section, we present the {\sf sGS-imiPADMM} algorithm for solving problem \eqref{problem:primal}.
We first recall the majorization technique used in \cite{LiMajADMM2016} and the indefinite proximal terms used in
\cite{chen18}.
Since the two convex functions $f$ and $g$ in problem \eqref{problem:primal} are assumed to be continuously differentiable with Lipschitz continuous gradients, there exist two self-adjoint positive semidefinite linear operators $\hat\Sigma_{f}:\cX\to\cX$ and $\hat\Sigma_{g}:\cY\to\cY$ such that
\[
\label{def-hat-fg}
\left\{
\begin{array}{ll}
f(x) \le \hat{f}(x;x^\prime):=
f({x}^\prime)+\langle \nabla f({x}^\prime),x-x^\prime\rangle
+\frac{1}{2}\|x-x^\prime\|_{\hat{\Sigma}_f}^2,&\forall x,\,x'\in\cX,
\\[2mm]
g(y) \le \hat{g}(y;y^\prime):=
g({y}^\prime)+\langle\nabla g({y}^\prime),y-y^\prime\rangle
+\frac{1}{2}\|y-{y}^\prime\|_{\hat{\Sigma}_g}^2,&\forall y,\,y'\in\cY.
\end{array}
\right.
\]
For any given $\sigma>0$, 
the \emph{majorized proximal augmented Lagrangian function} associated with problem \eqref{problem:primal} is defined by
\begin{multline*}
\label{majorizedalf}
\tilde\cL_\sigma(x,y;(x^{\prime},y^{\prime},z')):
= p(x)+\hat f(x;x^{\prime})+q(y)+\hat g(y;y^{\prime})
+\langle z',\cA^*x+\cB^*y-c\rangle
\\
\qquad\qquad\qquad
+\frac{\sigma}{2}\|\cA^*x+\cB^*y-c\|^2
+\frac{1}{2}\|x-x^\prime\|_{\tilde\cS}^2
+\frac{1}{2}\|y-y^\prime\|_{\tilde\cT}^2,
\\
\forall\,(x,y )\in\cX\times\cY\quad \mbox{and} \quad \forall\, (x^{\prime},y^{\prime},z^{\prime})\,\in\cX\times\cY\times\cZ,
\end{multline*}
where $\tilde\cS:\cX\to\cX$ and
$\tilde\cT:\cY\to\cY$ are self-adjoint ({\it not necessarily positive semidefinite}) linear operators satisfying
\[
\label{indcondition1}
\tilde\cS \succeq-\frac{1}{2}\hat\Sigma_f
\quad\mbox{and}\quad
\tilde\cT \succeq-\frac{1}{2}\hat\Sigma_g.
\]

In order to apply the block sGS decomposition theorem, we symbolically decompose the positive semidefinite linear operators $\hat\Sigma_{f}$ and $\hat\Sigma_{g}$ defined in \eqref{def-hat-fg} into the following form, i.e.,
\[
\label{decsig}
\hat\Sigma_{f}=\begin{pmatrix}
(\hat\Sigma_{f})_{11}&(\hat\Sigma_{f})_{12}&\cdots &(\hat\Sigma_{f})_{1m}\\[3mm]
(\hat\Sigma_{f})^*_{12}&(\hat\Sigma_{f})_{22}&\cdots &(\hat\Sigma_{f})_{2m}\\[3mm]
\vdots&\vdots&\ddots&\vdots \\[3mm]
(\hat\Sigma_{f})^*_{1m}&(\hat\Sigma_{f})^*_{2m}&\cdots &(\hat\Sigma_{f})_{mm}
\end{pmatrix}
\
\mbox{and}\
\hat\Sigma_{g}=\begin{pmatrix}
(\hat\Sigma_{g})_{11}&(\hat\Sigma_{g})_{12}&\cdots &(\hat\Sigma_{g})_{1n}\\[3mm]
(\hat\Sigma_{g})^*_{12}&(\hat\Sigma_{g})_{22}&\cdots &(\hat\Sigma_{g})_{2n}\\[3mm]
\vdots&\vdots&\ddots&\vdots \\[3mm]
(\hat\Sigma_{g})^*_{1n}&(\hat\Sigma_{g})^*_{2n}&\cdots &(\hat\Sigma_{g})_{nn}
\end{pmatrix},
\]
which are consistent with the decompositions of $\cX$ and $\cY$.
Meanwhile, the linear operators $\cA$ and $\cB$ are also can be decomposed as
$$
\cA z=(\cA_1 z,\ldots,\cA_m z)
\quad\mbox{and}\quad
\cB z=(\cB_1 z,\ldots,\cB_n z),
$$
where $\cA_i z\in\cX_i$ and $\cB_j z\in\cY_j$, $\forall i\in\{1,\ldots,m\}$, $j\in\{1,\ldots,n\}$, and $z\in\cZ$.
Moreover, we can also decompose the linear operators
$\tilde\cS:\cX\to\cX$ and
$\tilde\cT:\cY\to\cY$ as the decompositions of $\widehat\Sigma_f$ and $\widehat\Sigma_g$ in \eqref{decsig}.
To ensure the block sGS decomposition theorem, we require
 \[
\label{indcondition2}
(\hat\Sigma_{f})_{ii}+\sigma\cA_i\cA_i^*+\tilde \cS_{ii}\succ 0,
\quad i=1,\ldots,m,
\quad\mbox{and}\quad
(\hat\Sigma_{g})_{jj}+\sigma\cB_j\cB_j^*+\tilde \cT_{jj}\succ 0,
\quad j=1,\ldots,n,
\]
where
$\tilde\cS_{ii}:\cX_i\to\cX_i,\, i=1,\ldots,m$ and
$\tilde\cT_{jj}:\cY_j\to\cY_j,\,j=1,\ldots,n$ are the block diagonal parts of $\cS$ and $\cT$, respectively.
We now formally present the promised {\sf sGS-imiPADMM}.

\begin{algorithm}[!ht]
\caption{{\bf ({\sf sGS-imiPADMM}):} An inexact sGS decomposition based majorized indefinite-proximal ADMM for solving problem \eqref{problem:primal}.\label{algo}}
\normalsize
Let $\tau\in(0,(1+\sqrt{5})/2)$ be the step-length and $\{\tilde\varepsilon_k\}_{k\ge 0}$ be a summable sequence of nonnegative numbers.
Let $(x^0,y^0,z^0)\in\dom p\times\dom q\times\cZ$ be the initial point. For $k=0,1,\ldots$, perform the following steps:
\begin{description}
\item[\bf Step 1a. ({\tt Backward GS sweep})] Compute for $i=m,\ldots,2$,
$$
\begin{array}{l}
\disp
\tilde x_{i}^{k+1}\approx{\argmin_{x_{i}\in\cX_{i}}}
\left\{
\tilde{\cL}_\sigma\left( (x^{k}_{\le i-1},x_{i},\tilde x^{k+1}_{\ge i+1}),y^{k}; (x^k,y^k,z^k)\right)
\right\},
\\[4mm]
\disp
\tilde\delta^k_i\in\partial_{x_{i}}\tilde{\cL}_\sigma
\left( (x^{k}_{\le i-1},\tilde x_{i}^{k+1},\tilde x^{k+1}_{\ge i+1}),y^{k};(x^k,y^k,z^k)\right)
\ \mbox{with}\ \norm{\tilde\delta^k_i} \leq \tilde\varepsilon_k.
\end{array}
$$
\item[\bf Step 1b. ({\tt Forward GS sweep})] Compute for $i=1,\ldots,m$,
$$
\begin{array}{l}
\disp
x_{i}^{k+1} \approx {\argmin_{x_{i}\in\cX_{i}}} \left\{\tilde{\cL}_\sigma\left(
(x^{k+1}_{\le i-1},x_{i},\tilde x_{\ge i+1}^{k+1}),y^{k}; (x^k,y^k,z^k)\right)\right\},
\\[4mm]
\disp
\delta^k_i \in \partial_{x_{i}}\tilde{\cL}_\sigma\left(
(x^{k+1}_{\le i-1},x_{i}^{k+1},\tilde x_{\ge i+1}^{k+1}),y^{k}; (x^k,y^k,z^k)\right)\ \mbox{with}\ \norm{\delta_{i}^{k}} \leq \tilde\varepsilon_k.
\end{array}
$$
\item[\bf Step 2a. ({\tt Backward GS sweep})] Compute for $j=n,\ldots,2$,
$$
\begin{array}{l}
\disp
\tilde y_{j}^{k+1} \approx {\argmin_{y_{j}\in\cY_{j}}} \left\{\tilde{\cL}_\sigma\left(x^{k+1},
(y_{\le j-1}^{k},y_{j},\tilde y^{k+1}_{\ge j+1}); (x^k,y^k,z^k)\right)\right\},
\\[4mm]
\disp
\tilde\gamma^k_j \in \partial_{y_{j}}\tilde{\cL}_\sigma\left(x^{k+1},
(y_{\le j-1}^{k},\tilde y^{k+1}_{j},\tilde y^{k+1}_{\ge j+1}); (x^k,y^k,z^k)\right)\ \mbox{with}\ \norm{\tilde\gamma^k_j} \leq \tilde\varepsilon_k.
\end{array}
$$
\item[\bf Step 2b. ({\tt Forward GS sweep})] Compute for $j=1,\ldots,n$,
$$
\begin{array}{l}
\disp
y_{j}^{k+1} \approx{\argmin_{y_{j}\in\cY_{j}}}\left\{\tilde{\cL}_\sigma\left(x^{k+1},(y_{\le j-1}^{k+1},y_{j},\tilde y^{k+1}_{\ge j+1}); (x^k,y^k,z^k)\right)\right\},
\\[4mm]
\disp
\gamma^k_j \in \partial_{y_{j}}\tilde{\cL}_\sigma\left(x^{k+1},(y_{\le j-1}^{k+1},y^{k+1}_{j},\tilde y^{k+1}_{\ge j+1}); (x^k,y^k,z^k)\right) \ \mbox{with}\ \norm{\gamma^k_j} \leq\tilde\varepsilon_k.
\end{array}
$$
\item[\bf Step 3.] Compute $z^{k+1}:=z^k+\tau\sigma(\cA^*x^{k+1}+\cB^* y^{k+1}-c).$
\end{description}
\end{algorithm}

Following the discussions in \cite{chen}, here we define two linear operators $\tilde\cM:\cX\to\cX$ and $\tilde\cN:\cY\to\cY$ as follows:
\[
\label{eq-MN}
\left\{
\begin{array}{ll}
\tilde\cM:= \hat\Sigma_f + \sigma \cA\cA^* + \tilde\cS,
\\[2mm]
\tilde\cN:= \hat\Sigma_g + \sigma \cB\cB^* + \tilde\cT.
\end{array}
\right.
\]
Just like the decomposition of $\widehat\Sigma_f$ and $\widehat\Sigma_g$ in \eqref{decsig},
we can symbolically decompose $\tilde\cM$ and $\tilde\cN$ accordingly.
We use $\tilde\cM_d$ and $\tilde\cN_d$ to denote the corresponding diagonal parts, and $\tilde\cM_u$ and $\tilde\cN_u$ to denote the strictly upper triangular parts, respectively.
Consequently, $\tilde\cM=\tilde\cM_{d}+\tilde\cM_{u}+\tilde\cM_{u}^{*}$ and $\tilde\cN=\tilde\cN_{d}+\tilde\cN_{u}+\tilde\cN_{u}^{*}$.

\begin{remark}
Note that the linear operators $\tilde\cS:\cX\to\cX$ and $\tilde\cT:\cY\to\cY$ are chosen for the purpose of
compensating the deviation from the majorized augmented Lagrangian function to the original augmented Lagrangian function.
Meanwhile, they should be chosen such that the minimization subproblems involving $p_{1}$ and $q_{1}$ are easier to solve.
With appropriately chosen $\tilde\cS_{11}$ and $\tilde\cT_{11}$, we can assume that the following well-defined optimization problems
$$
\begin{array}{l}
\disp
\min\limits_{x_1\in\mathcal{X}_1}
\left\{ p_1(x_{1})+\frac{1}{2}\|x_{1}-x_{1}'\|^{2}_{\tilde\cM_{11}} \right\}
\quad\mbox{and}\quad
{\min\limits_{y_1\in\mathcal{Y}_1}} \left\{ q_1(y_{1})+\frac{1}{2}\|y_{1}-y_{1}'\|^{2}_{\tilde\cN_{11}} \right\}
\end{array}
$$
can be solved to a sufficient accuracy in the sense
of returning approximate solutions with sufficiently small subgradients of the objective functions, for any given $x_{1}'\in\cX_{1}$ and $y_{1}'\in\cY_{1}$.
\end{remark}

\color{black}
Recall that the Karush-Kuhn-Tucker (KKT) system of problem \eqref{problem:primal} is given by
\[
\label{eq:kkt}
0\in\partial p(x)+\nabla f(x)+\cA z,
\quad
0\in\partial q(y)+\nabla g(y)+\cB z,
\quad
\cA^* x+\cB^* y=c.
\]
If $(\bar x,\bar y,\bar z)\in\cX\times\cY\times\cZ$ satisfies \eqref{eq:kkt}, from \cite[Corollary 30.5.1]{rocbook} we know that $(\bar x, \bar y)$ is an optimal solution to problem \eqref{problem:primal} and $\bar z$ is an optimal solution to the dual of this problem.
To simplify the notation, we denote the solution set of the KKT system \eqref{eq:kkt} for problem \eqref{problem:primal} by $\overline\cW$.

We now make the following assumption on problem \eqref{problem:primal} and Algorithm \ref{algo}.
\begin{assumption}
\label{ass} Assume that:
\begin{itemize}[topsep=1pt,itemsep=-.6ex,partopsep=1ex,parsep=1ex]
\item[\rm (i)] the solution set $\overline\cW$ to the KKT system \eqref{eq:kkt} of problem \eqref{problem:primal} is nonempty;
\item[\rm (ii)] the self-adjoint positive semidefinite linear operators $\hat\Sigma_f:\cX\to\cX$ and $\hat\Sigma_g:\cY\to\cY$ are chosen such that \eqref{def-hat-fg} holds;
\item[\rm (iii)] the self-adjoint linear operators $\tilde\cS$ and $\tilde\cT$ are chosen such that \eqref{indcondition1} and \eqref{indcondition2} hold.
\color{black}
\end{itemize}
\end{assumption}

Under Assumption \ref{ass}, we can define the following linear operators:
\[
\label{st}
\left\{
\begin{array}{ll}
 \cS_{\rm sGS}:=\tilde\cS+{\rm sGS}(\tilde \cM)
=\tilde\cS+\tilde\cM_u\tilde\cM_d^{-1}\tilde\cM_u^*,
\\[2mm]
 \cT_{\rm sGS}:=\tilde\cT+{\rm sGS}(\tilde \cN)
=\tilde\cT+\tilde\cN_u\tilde\cN_d^{-1}\tilde\cN_u^*.
 \end{array}
\right.
\]
Based on the above preparations, the global convergence of Algorithm \ref{algo} is given as the following theorem. The corresponding proof will be accomplished in section \ref{sec:proof}.

\begin{theorem} [Convergence of {\sf sGS-imiPADMM}]
\label{thm:conv}
Suppose that Assumption \ref{ass} holds, and the linear operators $\widetilde\cS$ and $\widetilde\cT$ are chosen such that
\[
\label{psdcond2}
\frac{1}{2}\hat\Sigma_f+\sigma\cA\cA^*+\cS_{\rm sGS}\succ 0
\quad\mbox{and}\quad
\frac{1}{2}\hat\Sigma_g+\sigma\cB\cB^*+\cT_{\rm sGS}\succ 0.
\]
Then the whole sequence $\{(x^k,y^k,z^k)\}$ generated by Algorithm \ref{algo} converges to a solution of the KKT system \eqref{eq:kkt} of problem \eqref{problem:primal}.
\end{theorem}

We end this section by comparing Algorithm \ref{algo} and its convergence theorem
(Theorem \ref{thm:conv}) with its precursors in \cite{li2016schur,chen,zhangn} for solving problem \eqref{problem:primal}. Such a comparison will clearly demonstrate from where the algorithm in this paper originates and to what extent does the progress made in this paper can reach. The details of the comparison are presented in the following table.
\begin{center}
\begin{tabular}{|c|c|c|c|c|}
\hline&&&&\\[-3mm]
{\bf Ref.} $\backslash$ {\bf Item} & $f$ \& $g$. &Majorization &Proximal Terms & Inexact
\\[.5mm]
\hline&&&&
\\[-4mm]
\cite{li2016schur} &
separable, quadratic&
no & semidefinite & no
\\[1mm]
\hline&&&&
\\[-4mm]
\cite{chen} & - & yes & semidefinite & yes\\[2mm]
\hline&&&&
\\[-4mm]
\cite{zhangn} & separable & yes & indefinite & no
\\[1mm]
\hline&&&&
\\[-3mm]
This paper &
-
& yes & indefinite & yes
\\[1mm]
\hline
\end{tabular}
\end{center}
Here, the column ``$f$ \& $g$'' indicates the additional conditions
imposed on the functions $f$ and $g$ in problem \eqref{problem:primal},
the column ``Majorization'' indicates whether the majorization technique was used,
the column ``Proximal Terms'' shows whether the proximal terms used are semidefinite or indefinite,
and the column ``Inexact'' shows whether the subproblems are allowed to be solved approximately.
It is easy to conclude from the above table that Algorithm \ref{algo} proposed in this paper generalizes all those in \cite{li2016schur,chen,zhangn}.
This explains why we name the {\sf sGS-imiPADMM} as a unified algorithmic framework. Here, it is also worthwhile to point out that even when the proximal terms in {\sf sGS-imiPADMM} are chosen to be positive semidefinite, the convergence theorem in this paper is sharper than that in \cite{chen}; see Remark \ref{rmk43} for the details.

\section{Convergence Analysis}
\label{sec:proof}
In this section, we will prove Theorem \ref{thm:conv} step-by-step.
We first show how to apply the sGS decomposition theorem to
reformulate the multi-block Algorithm \ref{algo} as
an abstract $2$-block ADMM-type algorithm. Then we establish the convergence properties of the later, and, as a consequence, prove Theorem \ref{thm:conv}.

\subsection{Basic convergence results}
\begin{proposition}
\label{prop:equiv}
Suppose that \eqref{indcondition1} and \eqref{indcondition2} hold.
Define for all $k\ge 0$,
$$
\tilde\delta^{k}:=(\tilde\delta_1^{k},\ldots,\tilde\delta_{m}^{k}),\ \,
\delta^{k}:=(\delta_1^{k},\ldots,\delta_{m}^{k}),\ \,
\tilde\gamma^{k}:=(\tilde\gamma_1^{k},\ldots,\tilde\gamma_{n}^{k}),\
\mbox{and}\
\gamma^{k}:=(\gamma_1^{k},\ldots,\gamma_{n}^{k})$$
with the convention that
$\tilde\delta_{1}^{k}:=\delta_{1}^{k}$ and $\tilde\gamma_{1}^{k}:=\gamma_{1}^{k}$.
 Then
\begin{itemize}[topsep=1pt,itemsep=-.6ex,partopsep=1ex,parsep=1ex,leftmargin=4ex]
\item[\rm (i)]
 the sequences $\{(x^{k},y^{k},z^{k})\}$,
$\{\delta^{k}\}$, $\{\tilde\delta^{k}\}$, $\{\gamma^{k}\}$ and $\{\tilde\gamma^{k}\}$ generated by the {\sf sGS-imiPADMM} are well-defined;

\item[\rm (ii)] the linear operators $\cS_{\rm sGS}$ and $\cT_{\rm sGS}$ in \eqref{st} are well-defined, and
\[
\label{mnsgs}
\cM_{\rm sGS}:=\hat\Sigma_{f}+\sigma\cA\cA^{*} + \cS_{\rm sGS}\succ 0,
\quad
\cN_{\rm sGS}:=\hat\Sigma_{g}+\sigma\cB\cB^{*} + \cT_{\rm sGS}\succ 0;
\]
\item[\rm (iii)] it holds that
$$
\left\{
\begin{array}{l}
d_x^{k}
\in\partial_{x}\left\{
\tilde\cL_\sigma\left(\underuparrow{x}^{k+1},y^{k}; (x^k,y^k,z^k)\right)
+\frac{1}{2}\|\underuparrow{x}^{k+1}-x^k\|^{2}_{{\rm sGS}(\widetilde\cM)}
\right\},
\\[5mm]
d_y^{k}
\in\partial_{y}\left\{
\tilde\cL_\sigma\left(x^{{k+1}},\underuparrow{y}^{k+1}; (x^k,y^k,z^k)\right)
+\frac{1}{2}\|\underuparrow{y}^{k+1}-y^k\|^{2}_{{\rm sGS}(\widetilde\cN)}
\right\},
\end{array}\right.
$$
where
$$
d_x^{k}:=\delta^{k}+\tilde\cM_{u}\tilde\cM_{d}^{-1}(\delta^{k}-\tilde\delta^{k})
\quad\mbox{and}\quad
d_y^{k}:=\gamma^{k}+\tilde\cN_{u}\tilde\cN_{d}^{-1}(\gamma^{k}-\tilde\gamma^{k});
$$
\item[\rm (iv)] one has that
$\|\cM_{\rm sGS}^{-\frac{1}{2}}d_{x}^{k}\|\le\kappa \tilde \varepsilon_{k}$,
$\|\cN_{\rm sGS}^{-\frac{1}{2}}d_{y}^{k}\|\le\kappa' \tilde\varepsilon_{k}$, where
\[
\label{def:kappa}
\left\{
\begin{array}{l}
\kappa := 2\sqrt{m-1}\|\tilde\cM_{d}^{-\frac{1}{2}}\|
+\sqrt{m}\|\tilde\cM_{d}^{\frac{1}{2}}(\tilde\cM_{d}+\tilde\cM_{u})^{-1}\|,
\\[3mm]
\kappa' := 2\sqrt{n-1}\|\tilde\cN_{d}^{-\frac{1}{2}}\|
+\sqrt{n}\|\tilde\cN_{d}^{\frac{1}{2}}(\tilde\cN_{d}+\tilde\cN_{u})^{-1}\|.
\end{array}
\right.
\]
\end{itemize}
\end{proposition}
\begin{remark}
It is easy to prove Proposition \ref{prop:equiv} via Theorem \ref{prop:sgs}.
We omit the detailed proof here since it is almost the same as that of \cite[Proposition 3.1]{chen}.
\end{remark}

\subsection{An inexact $2$-block majorized indefinite-proxiaml ADMM}
\label{sec:imipadmm}
Based on Proposition \ref{prop:equiv} and the previous efforts (see e.g. \cite{lisgs,chen}), here we also view the {\sf sGS-imiPADMM} as a $2$-block ADMM-type algorithm applied to problem \eqref{problem:primal} with {\bf intelligently constructed proximal terms}. For this purpose, we formally present the previously mentioned {\sf imiPADMM} as Algorithm \ref{algo:imipadmm}, where the \emph{majorized augmented Lagrangian function} associated with problem \eqref{problem:primal} is defined by
\begin{multline*}
\label{majorizedalf}
\widehat\cL_\sigma(x,y;(x^{\prime},y^{\prime},z')):
= p(x)+\hat f(x;x^{\prime})+q(y)+\hat g(y;y^{\prime})
+\langle z',\cA^*x+\cB^*y-c\rangle
\\
\qquad\qquad
+\frac{\sigma}{2}\|\cA^*x+\cB^*y-c\|^2,\quad
\forall\,(x,y )\in\cX\times\cY\quad \mbox{and} \quad \forall\, (x^{\prime},y^{\prime},z^{\prime})\,\in\cX\times\cY\times\cZ.
\end{multline*}

\begin{algorithm}[!ht]
\caption{ {\sf (imiPADMM):} An inexact majorized indefinite-proximal ADMM for solving problem \eqref{problem:primal}.\label{algo:imipadmm}}
\normalsize
Let $\tau\in(0,(1+\sqrt{5})/2)$ be the step-length and $\{\varepsilon_k\}_{k\geq 0}$ be a summable sequence of nonnegative numbers.
Choose the self-adjoint (not necessarily positive semidefinite) linear operators $\cS$ and $\cT$ such that
\[
\label{ineq:succ0}
\cS\succeq -\frac{1}{2}\widehat\Sigma_f,
\quad
\cT\succeq -\frac{1}{2}\widehat\Sigma_g,
\quad
\frac{1}{2}\hat{\Sigma}_f+\sigma\cA\cA^{*}+\cS\succ 0
\quad
\mbox{and}\quad
\frac{1}{2}\hat\Sigma_{g}+\sigma\cB\cB^{*}+\cT\succ 0.
\]
For $k=0,1,\ldots$, perform the following steps:
\begin{description}
\item[\bf Step 1.] Compute $x^{k+1}$ and $d_x^k\in\partial \psi_k(x^{k+1})$ such that $\norm{(\hat\Sigma_{f}+\sigma\cA\cA^{*} + \cS)^{-\frac{1}{2}}d^k_x} \leq \varepsilon_k$,
where
\begin{equation}
 \label{barxplus}
x^{k+1} \approx \bar{x}^{k+1} :=
\argmin_{x\in\cX}\left\{
\psi_k(x):=
\hat{\cL}_\sigma\left(x,y^k; (x^k,y^k,z^k)\right)+\frac{1}{2}\norm{x-x^k}_{\cS}^2\right\}.
\end{equation}
\item[\bf Step 2.] Compute $y^{k+1}$ and $d_y^k\in\partial \varphi_k(y^{k+1})$ such that $\norm{(\hat\Sigma_{g}+\sigma\cB\cB^{*}+\cT)^{-\frac{1}{2}}d^k_y} \leq \varepsilon_k$,
where
\begin{align}
y^{k+1} \approx \bar{y}^{k+1}:&=
\argmin_{y\in\cY}\left\{\varphi_k(y):=
\hat{\cL}_\sigma\left({x}^{k+1},y; (x^k,y^k,z^k)\right)+\frac{1}{2}\norm{y-y^k}_{\cT}^2\right\}
\label{baryplus}.
\end{align}
\item[\bf Step 3.] Compute $z^{k+1}:=z^k+\tau\sigma(\cA^*x^{k+1}+\cB^* y^{k+1}-c)$.
\end{description}
\end{algorithm}

Now, we are ready to present the convergence theorem of the {\sf imiPADMM}.
The proof of this theorem is postponed to Appendix \ref{sec:convergence}.

\begin{theorem}[Convergence of {\sf imiPADMM}]
\label{theorem:convergence}
Suppose that parts {\rm (i)} and {\rm (ii)} in Assumption \ref{ass} hold.
Then the sequence $\{(x^k,y^k,z^k)\}$ generated by Algorithm \ref{algo:imipadmm} converges to a point in $\overline\cW$, i.e., the solution set to the KKT system \eqref{eq:kkt} of problem \eqref{problem:primal}.
\end{theorem}

\begin{remark}
Even though the purpose of the proposed Algorithm \ref{algo:imipadmm} is to derive the convergence properties of Algorithm \ref{algo}, this $2$-block ADMM-type algorithm itself is a very general extension of the classic ADMM that contains many contemporary practical techniques, including the original large dual step-lengths in \cite{glowinski1975approximation},
the positive semidefinite proximal terms in \cite{fazel2013hankel},
the majorized augmented Lagrangian function and indefinite proximal terms in \cite{LiMajADMM2016}, and
the error tolerance criteria in \cite{chen}.
\end{remark}

\begin{remark}\label{rmk43}
If both $\mathcal{S}$ and $\mathcal{T}$ are chosen as positive semidefinite linear operators, Algorithm \ref{algo:imipadmm} reduces to Algorithm {\sf imsPADMM} in \cite{chen}.
Moreover, if the subproblems \eqref{barxplus} and \eqref{baryplus} in Algorithm \ref{algo:imipadmm} are solved exactly, i.e., by restricting $\varepsilon_k\equiv 0,\,\forall k\ge 0$, {\sf imiPADMM} here reduces to the {\sf Majorized iPADMM} proposed in \cite{LiMajADMM2016}.
However, in both \cite{chen} and \cite{LiMajADMM2016}, one requires\footnote{The precise definitions of ${\Sigma}_f$ and ${\Sigma}_g$ are given in \eqref{def-fg-low} of the Appendix.}
$$
\frac{1}{2} {\Sigma}_f+\cS+\sigma\cA\cA^{*}\succ 0
\quad
\mbox{and}\quad
\frac{1}{2} \Sigma_{g}+\cT+\sigma\cB\cB^{*}\succ 0,
$$
where $\Sigma_f\preceq \hat\Sigma_f$ and $\Sigma_g\preceq\hat\Sigma_g$, and
these conditions are in general
stronger than the last two conditions in \eqref{ineq:succ0}.
Therefore, even for $2$-block problems,
Theorem \ref{theorem:convergence} on the convergence of Algorithm \ref{algo:imipadmm}
has made its own progress on improving the convergence properties of the previously proposed algorithms
{\sf imsPADMM} and {\sf Majorized iPADMM} in \cite{chen} and \cite{LiMajADMM2016}, respectively.
As a result, for the multi-block problem \eqref{problem:primal}, Theorem \ref{thm:conv} can also be used to sharpen the convergence properties of the {\sf sGS-imsPADMM} in \cite{chen}.
\end{remark}

\subsection{Convergence of the {\sf sGS-imiPADMM}}
Now we are ready to prove Theorem \ref{thm:conv} for Algorithm \ref{algo} based on the connection between the {\sf sGS-imiPADMM} for the multi-block problem \eqref{problem:primal} and the {\sf imiPADMM} for the same problem but from the angle of viewing it as a $2$-block problem.

\medskip
\begin{proof}[Proof of Theorem \ref{thm:conv}]
Suppose that Assumption \ref{ass} holds.
Let $\cS:=\cS_{\rm sGS}$ and $\cT:=\cT_{\rm sGS}$, where $\cS_{\rm sGS}$ and $\cT_{\rm sGS}$ are given in \eqref{st}.
According to \eqref{psdcond2} we know that \eqref{ineq:succ0} holds.
Moreover, one has from \eqref{mnsgs} that
$\hat\Sigma_{f}+\sigma\cA\cA^{*} + \cS=\cM_{\rm sGS}$
and
$\hat\Sigma_{g}+\sigma\cB\cB^{*} + \cT=\cN_{\rm sGS}$.
Thus by Proposition \ref{prop:equiv}(iii), one has that $d_x^k\in\partial \psi_k(x^{k+1})$ and $d_y^k\in\partial \varphi_k(y^{k+1})$.
Meanwhile, we can define the sequence $\{\varepsilon_{k}\}$ in Algorithm \ref{algo:imipadmm} by $\varepsilon_{k}:=\max\{\kappa,\kappa'\}\tilde\varepsilon_{k},\, \forall k\ge 0$, which is summable due to the fact that the sequence $\{\tilde\varepsilon_k\}$ used in Algorithm \ref{algo} to control the
inexactness is summable, where $\kappa$ and $\kappa'$ are given in \eqref{def:kappa}.
Then, by Proposition \ref{prop:equiv}(iv), one has that $\norm{\cM_{\rm sGS}^{-\frac{1}{2}}d^k_x} \leq \varepsilon_k$ and $\norm{\cN_{\rm sGS}^{-\frac{1}{2}}d^k_y} \leq \varepsilon_k$.
Thus, since Assumption \ref{ass}(iii) holds, the sequence $\{(x^k,y^k,z^k)\}$ generated by Algorithm \ref{algo} is {\bf exactly} a sequence generated by Algorithm \ref{algo:imipadmm} with the specially constructed proximal terms $\cS=\cS_{\rm sGS}$ and $\cT=\cT_{\rm sGS}$.
Consequently, since parts (i) and (ii) in Assumption \ref{ass} hold, from Theorem \ref{theorem:convergence} we know that Theorem \ref{thm:conv} holds. This completes the proof.
\qed
\end{proof}
\section{Conclusions}\label{sec:conclude}
In this paper, we have developed a unified algorithmic framework, i.e., {\sf sGS-imiPADMM}, for solving
the multi-block convex composite programming problem \eqref{problem:primal}.
The proposed algorithm combines the merits from its various precursors by gathering the practical techniques developed for the purpose of improving the efficiency of ADMM-type algorithms.
The motivation behind such a unification is that, these techniques, including the majorization-type surrogates, inexact symmetric Gauss-Seidel decomposition, indefinite proximal terms, inexact computation of subproblems, and large dual step-lengths, have been shown to be very useful in dealing with convex composite programming problems. We established the global convergence of the {\sf sGS-imiPADMM} under very mild assumptions. We believe that the proposed algorithm can serve not only as a generalization or extension of the existing algorithms, but also provide a catalyst for enhancing the numerical performance of multi-block ADMM based solvers. We should mention that the linear convergence rate of {\sf sGS-imiADMM} is not discussed in this paper, but one should be able to establish such results following the works conducted in \cite{HSZ18,zhangn} without much difficulty.

\vspace{2mm}
\noindent
{\bf Acknowledgements.} The authors would like to thank Prof. Xudong Li and Dr. Xiaoliang Song for some helpful discussions on Algorithm \ref{algo}.


\appendix
\section{Convergence Analysis of the {\sf imiPADMM}}
\label{sec:convergence}
In this part, we prove Theorem \ref{theorem:convergence}. We begin
by introducing the notation and definitions that will be used throughout this section.
Then we establish a key inequality for obtaining the convergence of the algorithm.
After that, we turn to the convergence of the {\sf imiPADMM}.

\subsection{Additional notation and preliminaries}
Recall that $\cH$ is a finite dimensional real Euclidean space endowed with an inner product $\langle \cdot,\cdot\rangle$ and its induced norm $\|\cdot\|$.
Then, we have that
\[
\label{eq:triangle}
\begin{array}{l}
\langle u,v\rangle_\cH=\frac{1}{2}\left(\|u\|_\cH^2+\|v\|_\cH^2-\|u-v\|_\cH^2\right)=\frac{1}{2}\left(\|u+v\|_\cH^2-\|u\|_\cH^2-\|v\|_\cH^2\right).
\end{array}
\]
The following lemma was given in \cite[Lemma 3.2]{zhangn}.
\begin{lemma}
\label{lemma:zhangning}
Let $h:\cU\to(-\infty,\infty)$ be a smooth convex function and there is a self-adjoint positive semidefinite linear operator $\widehat\Sigma_h:\cU\to\cU$ such that, for any given $u'\in\cU$,
$$
h(u)\le h(u')+\langle \nabla h(u'), u-u'\rangle+\frac{1}{2}\|u-u'\|^2_{\widehat\Sigma_h},
\quad\forall u\in\cU.
$$
Then it holds that for any given $u'\in\cU$,
\[
\label{ineqz}
\langle \nabla h(u)-\nabla h(u'),v-u'\rangle
\ge-\frac{1}{4}\|v-u\|^2_{\widehat\Sigma_h},
\quad\forall u,v\in{\cU}.
\]
\end{lemma}
We now turn to problem \eqref{problem:primal} and Algorithm \ref{algo:imipadmm}.
Due to the convexity of $f$ and $g$, there exist two positive semidefinite linear operators $\Sigma_f$ ($\preceq\hat{\Sigma}_f$) and $\Sigma_g$ ($\preceq \hat{\Sigma}_g$) such that
\[
\label{def-fg-low}
\begin{array}{l}
f(x)\ge
f({x}^\prime)+\langle \nabla f({x}^\prime),x-x^\prime\rangle
+\frac{1}{2}\|x-x^\prime\|_{{\Sigma}_f}^2,\quad\forall x,\,x'\in\cX,
\\[2mm]
g(y)\ge
g({y}^\prime)+\langle\nabla g({y}^\prime),y-y^\prime\rangle
+\frac{1}{2}\|y-{y}^\prime\|_{{\Sigma}_g}^2,\quad\forall \,y, \,y'\in\cY.
\end{array}
\]
Since that the sequence $\{\varepsilon_k\}$ in Algorithm \ref{algo:imipadmm} is nonnegative and summable, we can define the following two real numbers
$$
\cE:=\sum_{k=0}^{\infty}\varepsilon_{k}\quad
\mbox{and}\quad \cE':=\sum_{k=0}^{\infty}\varepsilon^{2}_{k}.
$$
Suppose that both {\rm (i)} and {\rm (ii)} in Assumption \ref{ass} hold.
Then, an infinite sequence $\{(x^{k},y^{k},z^{k})\}$ can be generated by Algorithm \ref{algo:imipadmm}.
Meanwhile, there exist two sequences $\{\bar x^{k}\}$ and $\{\bar y^{k}\}$ defined by \eqref{barxplus} and \eqref{baryplus}, respectively. In this case, we define for any $k\ge 0$,
$$
\label{denote}
\left\{
\begin{array}{lll}
r^{k}:=\cA^*x^{k}+\cB^*y^k-c,
&
\bar r^{k}:=\cA^*\bar x^{k}+\cB^*\bar y^{k}-c,
\\[2mm]
\tilde z^{k+1}:=z^{k}+\sigma r^{k+1},&
\bar z^{k+1}:=z^{k}+\tau\sigma\bar r^{k+1}
\end{array}
\right.
$$
with the convention that $\bar x^0=x^0$ and $\bar y^0=y^0$, where
$\tau$ is the step-length used in Algorithm \ref{algo:imipadmm}.
Moreover, we define the following three constants:
\[
\label{alpha}
\left\{
\begin{array}{ll}
\alpha:=(1+{\tau}/{\min\{1+\tau,1+\tau^{-1}\}})/2,
\\[2mm]
\hat\alpha:=1-\alpha\min\{\tau,\tau^{-1}\},
\\[2mm]
\beta:=\min\{1,1-\tau+\tau^{-1}\}\alpha-(1-\alpha)\tau.
\end{array}
\right.
\]
Based on the above definitions, we have the following result.
\begin{proposition}
\label{prop:error}
Suppose that both {\rm (i)} and {\rm (ii)} in Assumption \ref{ass} hold.
Let $\{(x^k,y^k,z^k)\}$ be the sequence generated by Algorithm \ref{algo:imipadmm}, and $\{\bar x^{k}\}$, $\{\bar y^{k}\}$ be the sequence defined by \eqref{barxplus} and \eqref{baryplus}.
Then, for any $k\ge 0$, we have that
$$\|x^{k+1}-\bar x^{k+1}\|_{\cM}
\le\varepsilon_k$$
and
$$
\|y^{k+1}-\bar y^{k+1}\|_{\cN}
\le(1+\sigma\|\cN^{-\frac{1}{2}}\cB\cA^{*}\cM^{-\frac{1}{2}}\|)\varepsilon_k,
$$
where
\[
\label{mnpd}
\cM:=\hat\Sigma_{f}+\sigma\cA\cA^{*} + \cS\succ 0
\quad \mbox{and}\quad
\cN:=\hat\Sigma_{g}+\sigma\cB\cB^{*} + \cT\succ 0.
\]
\end{proposition}
\begin{proof}
As a consequence of \eqref{ineq:succ0} in Algorithm \ref{algo:imipadmm}, \eqref{mnpd} holds.
Then, the subsequent proof can be easily completed via a few properties of the Moreau-Yosida mappings in \cite{lemar}, and one can refer to \cite[Proposition 3.1]{chen} and its proof for the details.
\qed
\end{proof}

The following result on a quasi-Fej\'er monotone sequence of real numbers will be used later.
\begin{lemma}
\label{lemma:sq-sum}
Let $\{a_k\}_{k\ge0}$ be a nonnegative sequence of real numbers satisfying $a_{k+1}\le a_k+\varepsilon_k$, $\forall k\ge 0$, where $\{\varepsilon_k\}_{k\ge 0}$ is a nonnegative and summable sequence of real numbers.
Then the {quasi-Fej\'er monotone sequence} $\{a_k\}$ converges to a unique limit point.
\end{lemma}

\subsection{The key inequality}
Now, we start to analyze the convergence of the {\sf imiPADMM} by studying some necessary results.
Let $\tau$ be the dual step-length in Algorithm \ref{algo:imipadmm} and $\alpha$ be the constant defined in \eqref{alpha}.
We define the following two linear operators
\[
\label{def:FG}
{\cF}:=\frac{1}{2}{\hat\Sigma_{f}}+\cS+\frac{(1-\alpha)\sigma}{2}\cA\cA^{*}\quad \mbox{ and}\quad
{\cG}:=\frac{1}{2}{\hat\Sigma_{g}}+\cT+\min\{\tau,1+\tau-\tau^2\}\alpha\sigma\cB\cB^{*},
\]
where $\hat\Sigma_{f}$ and $\hat\Sigma_{g}$ are given by \eqref{def-hat-fg}.
\begin{lemma}
\label{lem:fgpd}
Assume that \eqref{ineq:succ0} holds. For any $\tau\in(0,(1+\sqrt{5})/2)$, the constants $\alpha$, $\hat\alpha$ and $\beta$ defined by \eqref{alpha} satisfy
$0<\alpha<1$, $0<\hat\alpha<1$ and $\beta>0$.
Meanwhile,
the linear operators $\cF$ and $\cG$ defined by \eqref{def:FG} are positive definite.
\end{lemma}
\begin{proof}
It is easy to see that $0<\alpha<1$, $0<\hat\alpha<1$ and $\beta>0$ from \eqref{alpha} and the fact that $\tau\in(0,(1+\sqrt{5})/2)$.
Also, it holds that $\rho := \min (\tau, 1+\tau -\tau^2) \in (0,1]$ so that $0<\rho\alpha<1$.
Note that by \eqref{def:FG} we have that
$$
\begin{array}{ll}
 \cF&
 = \frac{(1-\alpha)}{2} \left( \frac{1}{2}\hat\Sigma_f + \cS + \sigma\cA\cA^*\right)
 +\frac{1+\alpha}{2}\left(
 \frac{1}{2}\hat\Sigma_f + \cS
 \right),
\\[4mm]
\cG&
=\frac{\rho\alpha}{2}\left( \frac{1}{2}\hat\Sigma_g + \cT + \sigma\cB\cB^*\right)
+\frac{2-\rho\alpha}{2}\left(\frac{1}{2}\hat\Sigma_g +\cT \right)
+\frac{\rho\alpha}{2}\sigma\cB\cB^*.
\end{array}
$$
Hence, one can readily observe that $\cF\succ 0$ and $\cG\succ 0$ from \eqref{ineq:succ0}. This completes the proof.
\qed
\end{proof}

Based on the previous results, one can get the following result, which is exactly the same as \cite[Theorem 5.2]{chen}. So we omit the corresponding proof.
\begin{lemma}
\label{lemma:4.1}
Suppose that both {\rm (i)} and {\rm (ii)} in Assumption \ref{ass} hold.
Then, the infinite sequence $\{(x^k,y^k,z^k)\}$ generated by Algorithm \ref{algo:imipadmm} satisfies, for all $k\ge1$,
\[
\label{ineq:rr}
\begin{array}{l}
(1-\tau)\sigma\|r^{k+1}\|^2
+\sigma\|\cA^*x^{k+1}+\cB^*y^k-c\|^2+2\alpha\inprod{d_y^{k-1}-d_y^k}{y^k-y^{k+1}}
\\[2mm]
\ge \;
\hat{\alpha}\sigma(\| r^{k+1}\|^{2}-\|r^{k}\|^{2})
+\beta \sigma\norm{r^{k+1}}^{2}
+\norm{x^{k}-x^{k+1}}_{\frac{(1-\alpha)\sigma}{2}\cA\cA^{*}}^{2}
\\[2mm]
\quad
-\|y^{k-1}-y^k\|_{\alpha(\hat\Sigma_g+\cT)}^2
+\|y^k-y^{k+1}\|_{\alpha(\hat\Sigma_g+\cT)+\min\{\tau,1+\tau-\tau^2\}\alpha\sigma\cB\cB^{*}}^2.
\end{array}
\]
\end{lemma}
Next, we shall derive an inequality which is essential for
establishing the global convergence of the {\sf imiPADMM}.
\begin{proposition}[The key inequality]
\label{prop:inequality}
Suppose that both {\rm (i)} and {\rm (ii)} in Assumption \ref{ass} hold.
For any $\bar w:=(\bar x,\bar y,\bar z)\in\overline\cW$, the sequence $\{(x^k,y^k,z^k)\}$
generated by Algorithm \ref{algo:imipadmm} satisfies
\[
\label{ineq:mainb}
\begin{array}{l}
\disp
2\alpha\inprod{d_y^k-d_y^{k-1}}{y^k-y^{k+1}}
-2\langle d_{x}^{k},x^{k+1}-\bar x\rangle
-2\inprod{ d_{y}^k}{y^{k+1}-\bar{y}}
\\[2mm]
\disp
+\|x^{k}-x^{k+1}\|_{\cF}^{2}
+\|y^{k}-y^{k+1}\|_{\cG}^{2}
+\beta \sigma\norm{r^{k+1}}^{2}
\le
\phi_k(\bar w) - \phi_{k+1}(\bar w),\quad\forall k\ge 1,
\end{array}
\]
where, for any $(x,y,z)\in\cX\times\cY\times\cZ$ and $k\ge 1$,
\[
\label{def:phik}
\begin{array}{ll}
\phi_k(x,y,z):=
&\disp\frac{1}{\tau\sigma}\|z-z^{k}\|^{2}
+\|x-x^{k}\|_{\hat\Sigma_{f}+\cS}^{2}
+\|y-y^{k}\|_{\hat\Sigma_{g}+\cT}^{2}
\\[2mm]
\disp
&+\sigma\|\cA^*x+\cB^*y^{k}-c\|^{2}
+\hat\alpha\sigma\norm{r^k}^{2}
+\alpha\norm{y^{k-1}-y^k}_{\hat\Sigma_g+\cT}^2.
\end{array}
\]
\end{proposition}
\begin{proof}
For any given $(x,y,z)\in\cX\times\cY\times\cZ$, we define $x_e:=x-\bar x$, $y_e:=y-\bar y$ and $z_e:=z-\bar z$.
From \eqref{denote} one has that
$$z^{k}+\sigma(\cA^*x^{k+1}+\cB^*y^{k}-c)
=\tilde z^{\,k+1}+\sigma\cB^{*}(y^{k}-y^{k+1}).$$
Then, from Step 1 of Algorithm \ref{algo:imipadmm}, we know that
\[
\label{inclu:subx}
d_{x}^{k}-\nabla f(x^{k})
-\cA\left(\tilde z^{\,k+1}+\sigma\cB^{*}(y^{k}-y^{k+1})\right)
+(\hat\Sigma_{f}+\cS)(x^{k}-x^{k+1})
\in
\partial p(x^{k+1}).
\]
Moreover, the convexity of $p$ implies that
\begin{multline}
\label{ineq:p(x)}
p(\bar x)
+\inprod{d_{x}^{k}-\nabla f(x^{k})
-\cA\left(\tilde z^{k+1}+\sigma\cB^{*}(y^{k}-y^{k+1})\right)
+(\hat\Sigma_{f}+\cS)(x^{k}-x^{k+1})}{x_e^{k+1}}
\ge
p(x^{k+1}).
\end{multline}
Applying a similar derivation, we can also get that for any $ y\in\cY$,
\[
\label{ineq:q(y)+g(y)2}
\begin{array}{l}
q(\bar y)+\inprod{ d_{y}^k}{y_e^{k+1}}
-\langle \nabla g(y^k),y_e^{k+1}\rangle
 -\inprod{\tilde z^{k+1}}{\cB^{*}y_e^{k+1}}
\\[2mm]
+\inprod{(\hat\Sigma_{g}+\cT)(y^{k}-y^{k+1})}{y_e^{k+1}}
\ge q(y^{k+1}).
\end{array}
\]
By using (\ref{eq:kkt}) and the convexity of the functions $p$ and $q$, we have
\begin{eqnarray}
\label{ineq:fgkp1ss}
\begin{array}{l}
p(x^{k+1})-p(\bar{x})
+\inprod{\nabla f(\bar x)+\cA\bar{z}}{{x}^{k+1}_e}
\geq
0,
\\[2mm]
q(y^{k+1})-q(\bar y)
+\inprod{\nabla g(\bar y)+\cB\bar{z}}{{y}^{k+1}_e}
\geq 0.
\end{array}
\end{eqnarray}
Finally, by summing \eqref{ineq:p(x)} \eqref{ineq:q(y)+g(y)2} and \eqref{ineq:fgkp1ss} together, we get
\[
\label{ineq:pq2}
\begin{array}{l}
-\langle \tilde z_e^{k+1} ,r^{k+1}\rangle
-\sigma\langle \cB^{*}(y^{k}-y^{k+1}) ,\cA^* x_e^{k+1}\rangle
\\[2mm]
+\langle x_e^{k+1},(x^{k}-x^{k+1})\rangle_{\hat\Sigma_{f}+\cS}
+\langle y_e^{k+1},(y^{k}-y^{k+1})\rangle_{\hat\Sigma_{g}+\cT}
+\langle d_{x}^{k},x_e^{k+1}\rangle
+\inprod{ d_{y}^k}{y_e^{k+1}}
\\[2mm]
\ge
\langle \nabla f(x^{k})-\nabla f(\bar x),x_e^{k+1}\rangle
+\langle \nabla g(y^k)-\nabla g(\bar y),y_e^{k+1}\rangle
\\[2mm]
\ge
-\frac{1}{4}\|x^{k}-x^{k+1}\|_{\widehat\Sigma_f}^2
-\frac{1}{4}\|y^{k}-y^{k+1}\|_{\widehat\Sigma_g}^2,
\end{array}
\]
where the last inequality comes from Lemma \ref{lemma:zhangning}.
Next, we estimate the left-hand side of
\eqref{ineq:pq2}.
By using \eqref{eq:triangle}, we have that
\[
\label{eq:AxBy}
\begin{array}{ll}
\inprod{\cB^{*}(y^{k}-y^{k+1})}{\cA^*x_e^{k+1}}
=\inprod{\cB^{*}y_e^{k}-\cB^{*} y_e^{k+1}}{r^{k+1}-\cB^*y_e^{k+1}}
\\[3mm]
=\inprod{\cB^{*}y_e^{k}-\cB^{*} y_e^{k+1}}{r^{k+1}}
-\frac{1}{2}\left(
\|\cB^*y_e^{k}\|^{2}
-\|\cB^{*}y_e^{k}-\cB^{*} y_e^{k+1}\|^2
-\|\cB^{*} y_e^{k+1}\|^2
\right)
\\[3mm]
=\frac{1}{2}\left(\|\cB^*y_e^{k+1}\|^{2}
+\|\cA^* x^{k+1}+\cB^*y^{k}-c\|^{2}
-\|\cB^*y_e^{k}\|^{2}
-\|r^{k+1}\|^{2}\right).
\end{array}
\]
Also, from \eqref{eq:triangle} we know that
\begin{eqnarray}
\label{eq:xxyy1}
\begin{array}{l}
\langle x_e^{k+1},x^{k}-x^{k+1}\rangle_{\hat\Sigma_{f}+\cS}
=
\frac{1}{2}(\|x_e^{k}\|_{\hat\Sigma_{f}+\cS}^{2}
-\|x_e^{k+1}\|_{\hat\Sigma_{f}+\cS}^{2})
-\frac{1}{2}\|x^{k}-x^{k+1}\|_{\hat\Sigma_{f}+\cS}^{2},
\\[3mm]
\langle y_e^{k+1},y^{k}-y^{k+1}\rangle_{\hat\Sigma_{g}+\cT}
=\frac{1}{2}(
\|y_e^{k}\|_{\hat\Sigma_{g}+\cT}^{2}
-\|y_e^{k+1}\|_{\hat\Sigma_{g}+\cT}^{2})
-\frac{1}{2}\|y^{k}-y^{k+1}\|_{\hat\Sigma_{g}+\cT}^{2}\, .
\end{array}
\end{eqnarray}
Moreover, by using the definition of $\{\tilde z^k\}$ and \eqref{eq:triangle} we know that
\[
\label{eq:rz}
\begin{array}{l}
\disp
\langle r^{k+1},\tilde z_e^{k+1}\rangle
=\inprod{r^{k+1}}{z_e^{k}+\sigma r^{k+1}}
=
\frac{1}{\tau\sigma}\langle z^{\,k+1}-z^{k},z_e^{k}\rangle
+\sigma\|r^{k+1}\|^{2}
\\[2mm]
=
\frac{1}{2\tau\sigma}\left(
\|z^{k+1}_e\|^{2}
-\|z^{k+1}-z^{k}\|^{2}
-\|z_e^{k}\|^{2}
\right)
+\sigma\|r^{k+1}\|^{2}
\\[2mm]
=
\frac{1}{2\tau\sigma}\left(
\|z^{k+1}_e\|^{2}-\|z_e^{k}\|^{2}
\right)
+\frac{(2-\tau)\sigma}{2}\|r^{k+1}\|^{2}.
\end{array}
\]
Thus, by substituting \eqref{eq:AxBy}, \eqref{eq:xxyy1} and \eqref{eq:rz} into \eqref{ineq:pq2}, we obtain that
\[
\label{ineq:pq2xxxx}
\begin{array}{l}
\langle d_{x}^{k},x_e^{k+1}\rangle
+\inprod{ d_{y}^k}{y_e^{k+1}}
+\frac{1}{2\tau\sigma}\left(
\|z_e^{k}\|^{2}-\|z^{k+1}_e\|^{2}
\right)
+\frac{\sigma}{2}(\|\cB^*y_e^{k}\|^{2}-\|\cB^*y_e^{k+1}\|^{2})
\\[3mm]
+\frac{1}{2}(\|x_e^{k}\|_{\hat\Sigma_{f}+\cS}^{2}
+\|y_e^{k}\|_{\hat\Sigma_{g}+\cT}^{2})
-\frac{1}{2}(
\|x_e^{k+1}\|_{\hat\Sigma_{f}+\cS}^{2}
+\|y_e^{k+1}\|_{\hat\Sigma_{g}+\cT}^{2})
\\[3mm]
\ge
\frac{1}{2}\|x^{k}-x^{k+1}\|_{\frac{1}{2}\hat\Sigma_{f}+\cS}^{2}
+\frac{1}{2}\|y^{k}-y^{k+1}\|_{\frac{1}{2}\hat\Sigma_{g}+\cT}^{2}
\\[2.5mm]
\quad+\frac{\sigma}{2}\|\cA^*x^{k+1}+\cB^*y^{k}-c\|^{2}
+\frac{(1-\tau)\sigma}{2}\|r^{k+1}\|^{2}.
\end{array}
\]
Note that for any $y\in\cY$, $\cA^*\bar x+\cB^*y-c=\cB^{*}y_{e}$.
Therefore, by applying \eqref{ineq:rr} in Lemma \ref{lemma:4.1} to the right hand side of \eqref{ineq:pq2xxxx} and using \eqref{def:phik} together with \eqref{def:FG}, we know that \eqref{ineq:mainb} holds for $k\ge 1$. This completes the proof. \qed
\end{proof}
\begin{remark}
The inequality \eqref{ineq:pq2xxxx} in the proof is responsible for the improvement that we made in this paper, when compared with \cite{chen}, in which the same problem as \eqref{problem:primal} was considered and the inequality for the same purpose as \eqref{ineq:pq2xxxx} is given as follows:
\[
\label{ineqori}
\begin{array}{l}
\inprod{ d_{x}^{k}}{x_e^{k+1}}+\inprod{ d_{y}^k}{y_e^{k+1}}
+\frac{1}{2\tau\sigma}
(\|z_e^{k}\|^{2}-\|z^{k+1}_e\|^{2})
+\frac{\sigma}{2}(\|\cB^*y_e^{k}\|^{2}-\|\cB^*y_e^{k+1}\|^{2})
\\[2mm]
+\frac{1}{2}
(\|x_e^{k}\|_{\hat\Sigma_{f}+\cS}^{2}+\|y_e^{k}\|_{\hat\Sigma_{g}+\cT}^{2})
-\frac{1}{2}(\|x_e^{k+1}\|_{\hat\Sigma_{f}+\cS}^{2}+\|y_e^{k+1}\|_{\hat\Sigma_{g}+\cT}^{2})
\\[2mm]
\ge
\boxed{
\frac{1}{2}\norm{x^k-x^{k+1}}^{2}_{\frac{1}{2}\Sigma_f+\cS}
+\frac{1}{2}\norm{y^k-y^{k+1}}^{2}_{\frac{1}{2}\Sigma_g+\cT}
}
\\[3mm]
\disp\quad+\frac{\sigma}{2}\|\cA^*x^{k+1}+\cB^*y^{k}-c\|^{2}
+\frac{(1-\tau)\sigma}{2}\|r^{k+1}\|^{2},
\end{array}
\]
where ${\Sigma_{f}}$ and ${\Sigma_{g}}$ are defined by \eqref{def-fg-low}.
The difference between \eqref{ineqori} and \eqref{ineq:pq2xxxx} is highlighted in a box. Since that $\hat\Sigma_f\succeq\Sigma_f$ and $\hat\Sigma_g\succeq\Sigma_g$, the inequality \eqref{ineq:pq2xxxx} is tighter than \eqref{ineqori}.
Consequently, the inequality \eqref{ineq:mainb} looks the same as \cite[(5.14) in Proposition 5.1]{chen}, but one should notice that the definitions of $\cF$ and $\cG$ in this paper are different from those in \cite{chen}, as can be seen from the following table.
\begin{center}
\begin{tabular}{|c|c|c|}
\hline&&\\[-3mm]
{\bf Ref.} $\backslash$ {\bf Item} & $\cF$ & $\cG$
\\[.5mm]\hline&&
\\[-3mm]
\cite{chen} & $\frac{1}{2}{\Sigma_{f}}+\cS+\frac{(1-\alpha)\sigma}{2}\cA\cA^{*}$ & $\frac{1}{2}{\Sigma_{g}}+\cT+\min\{\tau,1+\tau-\tau^2\}\alpha\sigma\cB\cB^{*}$
\\[2mm]\hline&&
\\[-3mm]
This paper &$\frac{1}{2} {\hat\Sigma_{f}}+\cS+\frac{(1-\alpha)\sigma}{2}\cA\cA^{*}$
 & $\frac{1}{2} {\hat\Sigma_{g}}+\cT+\min\{\tau,1+\tau-\tau^2\}\alpha\sigma\cB\cB^{*}$
\\[2mm]
\hline
\end{tabular}
\end{center}
\end{remark}

\subsection{Proof of Theorem \ref{theorem:convergence}}

The proof for Theorem \ref{theorem:convergence} can be obtained by using the newly defined $\cF$ and $\cG$ in \eqref{def:FG} instead of those used in \cite{chen} and repeating the proof of \cite[Theorem 5.1]{chen}. In order to make this paper more readable, we provide the proof of Theorem \ref{theorem:convergence} here.

\begin{proof}
According to Assumption \ref{ass}{\rm (i)} we know that the solution set $\overline\cW$ to the KKT system \eqref{eq:kkt} of problem \eqref{problem:primal} is nonempty. Then, we can
choose a fixed $\bar w:=(\bar x,\bar y,\bar z)\in\overline\cW$, and define $x_e:=x-\bar x$, $y_e:=y-\bar y$ and $z_e:=z-\bar z$ for any given $(x,y,z)\in\cX\times\cY\times\cZ$.

We first show that the sequence $\{(x^{k},y^k,z^k)\}$ is bounded.
According to Lemma \ref{lem:fgpd} one has that $0<\alpha<1$, $0<\hat\alpha<1$ and $\beta>0$.
Moreover, the linear operators $\cF$ and $\cG$ defined in \eqref{def:FG} are positive definite.
Hence, it holds that
\begin{multline*}
\qquad\qquad
\norm{y^k-y^{k+1}}^{2}_{\cG}
+2\alpha\inprod{ d_y^k- d_y^{k-1}}{y^k-y^{k+1}}
\\
= \norm{y^k-y^{k+1}+\alpha\cG^{-1}(d^k_y-d^{k-1}_y)}_\cG^2
-\alpha^2\norm{d^k_y-d^{k-1}_y}_{\cG^{-1}}^2.
\qquad\qquad
\end{multline*}
By substituting $\bar x^{k+1}$ and $\bar y^{k+1}$ for $x^{k+1}$ and $y^{k+1}$ into \eqref{ineq:mainb} in Proposition \ref{prop:inequality}, we obtain that
\begin{equation}
\label{ineq:contraction2}
\begin{array}{l}
\phi_{k}(\bar w)
-\bar \phi_{k+1}(\bar w)
+\alpha^{2}\|d_y^{k-1}\|^{2}_{\cG^{-1}}
\\[2mm]
\ge
\|\bar x^{k+1}-x^{k}\|^{2}_{\cF}
+\beta\sigma\|\bar r^{k+1}\|^{2}
+\|\bar y^{k+1}-y^{k}+\alpha\cG^{-1}d_y^{k-1}\|_{\cG}^{2},\quad\forall k\ge 1,
\end{array}
\end{equation}
where for any $(x,y,z)\in\cX\times\cY\times\cZ$, we define
$$
\begin{array}{ll}
\bar\phi_k(x,y,z):=
&
\frac{1}{\tau\sigma}\|z-\bar z^{k}\|^{2}
+\|x-\bar x^{k}\|_{\hat\Sigma_{f}+\cS}^{2}
+\|y-\bar y^{k}\|_{\hat\Sigma_{g}+\cT}^{2}
\\[3mm]
&\disp+\sigma\|\cA^*x+\cB^*\bar y^{k} -c\|^{2}
+\hat\alpha\sigma\norm{\bar{r}^k}^{2}
+\alpha\norm{y^{k-1}-\bar y^{k}}_{\hat\Sigma_g+\cT}^2,\quad\forall k\ge 1.
\end{array}
$$
Now, define the sequences $\{\xi^{k}\}$ and $\{\bar \xi^{k}\}$ in $\cZ\times\cX\times\cY\times\cZ\times\cY$
for $k\ge 1$ by
$$
\left\{
\begin{array}{l}
\xi^{k}:=\left(
\sqrt{\tau\sigma}
z^{k}_{e},
(\hat\Sigma_{f}+\cS)^{\frac{1}{2}}x^{k}_{e},
\cN^{\frac{1}{2}}y^{k}_{e},
\sqrt{\hat\alpha\sigma}r^{k},
\sqrt{\alpha}(\hat\Sigma_g+\cT)^{\frac{1}{2}}(y^{k-1}-y^k)
\right),
\\[3mm]
\bar\xi^{k}:=\left(
\sqrt{\tau\sigma}\bar z^{k}_{e},
(\hat\Sigma_{f}+\cS)^{\frac{1}{2}}\bar x^{k}_{e},\cN^{\frac{1}{2}}\bar y^{k}_e,
\sqrt{\hat\alpha\sigma}\bar r^{k},
\sqrt{\alpha}(\hat\Sigma_g+\cT)^{\frac{1}{2}}(y^{k-1}-\bar y^{k})
\right).
\end{array}
\right.
$$
Obviously we have $\|\xi_{k}\|^{2}=\phi_{k}(\bar w)$ and $\|\bar \xi_{k}\|^{2}=\bar \phi_{k}(\bar w)$.
This, together with
\eqref{ineq:contraction2} implies that
$
\|\bar\xi^{k+1}\|^{2}
\le
\|\xi^{k}\|^{2}
+\alpha^{2}\|\cG^{-\frac{1}{2}} d_y^{k-1}\|^{2}
$.
As a result, it holds that
$\|\bar\xi^{k+1}\|
\le\|\xi^{k}\|+\alpha\|\cG^{-\frac{1}{2}} d_y^{k-1}\|$.
Consequently, one obtains that
\[
\label{ineq:xik+1}
\|\xi^{k+1}\|
\le
\|\xi^{k}\|
+\alpha\|\cG^{-\frac{1}{2}} d_y^{k-1}\|+
\|\bar\xi^{k+1}-\xi^{k+1}\|.
\]
Next, we estimate $\|\bar\xi^{k+1}-\xi^{k+1}\|$ in \eqref{ineq:xik+1}.
From Lemma \ref{lem:fgpd} we know that $\hat{\alpha} + \tau \in [1,2]$, so that
$$
\begin{array}{l}
\frac{1}{\tau\sigma}\norm{\bar z^{k+1}-z^{k+1}}^2
+{\hat\alpha\sigma}\norm{\bar r^{k+1}-r^{k+1}}^2
=(\tau+\hat{\alpha})\sigma\norm{\bar r^{k+1}-r^{k+1}}^2
\\[2mm]
\le
2\sigma
\norm{\cA^*(\bar x^{k+1}-x^{k+1})+\cB^*(\bar y^{k+1}-y^{k+1})}^2
\le
4\norm{\bar x^{k+1}-x^{k+1}}_{\sigma\cA\cA^*}^{2}
+4\norm{\bar y^{k+1}-y^{k+1}}_{\sigma\cB\cB^{*}}^2.
\end{array}
$$
This, together with Proposition \ref{prop:error}, implies that
\[
\label{ineq:xixi}
\begin{array}{l}
\norm{\bar \xi^{k+1}-\xi^{k+1}}^2
\\[2mm]
\le
\norm{\bar x^{k+1}-x^{k+1}}^2_{\hat\Sigma_{f}+\cS}
+\norm{\bar y^{k+1}-y^{k+1}}^2_{\cN}
+\norm{\bar y^{k+1}-y^{k+1}}^2_{\hat\Sigma_g+\cT}
\\[2mm]
\quad +4\norm{\bar x^{k+1}-x^{k+1}}_{\sigma\cA\cA^*}^{2}
+4\norm{\bar y^{k+1}-y^{k+1}}_{\sigma\cB\cB^{*}}^2
\\[2mm]
\leq
5(\norm{\bar x^{k+1}-x^{k+1}}_{\cM}^{2}
+\norm{\bar y^{k+1}-y^{k+1}}_{\cN}^{2})
\le
\varrho^2 \varepsilon_k^{2},
\end{array}
\]
where $\varrho$ is a positive constant defined by
$
\varrho:=\sqrt{5(1+(1+\sigma\|\cN^{-\frac{1}{2}}\cB\cA^{*}\cM^{-\frac{1}{2}}\|)^2)}$.
On the other hand, from Proposition \ref{prop:error}, we know that $\norm{\cG^{-\frac{1}{2}} d_y^{k}}\le\norm{\cG^{-\frac{1}{2}}\cN^{\frac{1}{2}}} \varepsilon_{k}$.
By using this fact together with \eqref{ineq:xik+1} and \eqref{ineq:xixi}, one can get
\[
\label{ineq:xibound}
\begin{array}{rl}
\norm{\xi^{k+1}}
\le \norm{\xi^k}
+\varrho\varepsilon_{k}
+\norm{\cG^{-\frac{1}{2}}\cN^{\frac{1}{2}}}\varepsilon_{k-1}
\le\|\xi^{1}\|+\left(\varrho+ \norm{\cG^{-\frac{1}{2}}\cN^{\frac{1}{2}}}\right)\cE,
\end{array}
\]
which implies that the sequence $\{\xi^{k}\}$ is bounded. Hence by \eqref{ineq:xixi}, we know that the sequence $\{\bar \xi^{k}\}$ is also bounded.
From the definition of $\xi^{k}$, we know that the sequences $\{y^{k}\}$, $\{z^{k}\}$, $\{ r^k\}$ and $\{(\hat\Sigma_{f}+\cS)^{\frac{1}{2}}x^{k}\}$ are bounded.
Thus, by the definition of $r^k$, we know that the sequence $\{Ax^k\}$ is also bounded, which together with the definition of $\cM$ and the fact that $\cM\succ 0$, implies that $\{x^k\}$ is bounded.

By \eqref{ineq:contraction2}, \eqref{ineq:xixi} and \eqref{ineq:xibound} we have that
\[
\label{ineq:summable}
\begin{array}{l}
\disp
\sum_{k=1}^\infty
\left(
\norm{\bar{x}^{k+1}-x^{k}}^{2}_{\cF}
+\beta\sigma\norm{\bar r^{k+1}}^{2}
+\norm{\bar{y}^{k+1}-y^{k}+\alpha\cG^{-1}d_y^{k-1}}_{\cG}^{2}\right)
\\[3mm]
\disp
\leq
\sum_{k=1}^\infty \left(\phi_{k}(\bar w) - \phi_{k+1}(\bar w)
+\phi_{k+1}(\bar w) -\bar \phi_{k+1}(\bar w)
+\alpha^{2}\|d_y^{k-1}\|^{2}_{\cG^{-1}}\right)
\\[3.5mm]
\disp
\leq
\phi_{1}(\bar w)
+\mbox{$\sum_{k=1}^\infty$} \norm{\xi^{k+1}-\bar{\xi}^{k+1}}
\left(\norm{\xi^{k+1}}+\norm{\bar{\xi}^{k+1}}\right)
+\norm{\cG^{-\frac{1}{2}}\cN^{\frac{1}{2}}}^2\cE'
\\[2mm]
\disp
\leq
\phi_{1}(\bar w)
+\norm{\cG^{-\frac{1}{2}}\cN^{\frac{1}{2}}}^2\cE'
+\varrho \, { \max\limits_{k\geq 1} }\{ \norm{\xi^{k+1}}+\norm{\bar{\xi}^{k+1}}\}\cE
<\infty,
\end{array}
\]
where we have used the fact that $\phi_{k}(\bar w) -\bar \phi_{k}(\bar w)
\leq \norm{\xi^{k}-\bar{\xi}^{k}}(\norm{\xi^{k}}+\norm{\bar{\xi}^{k}})$.
By \eqref{ineq:summable}, we know that
$\lim_{k\to\infty}\norm{\bar{x}^{k+1}-x^{k}}^{2}_{\cF}= 0$,
$\lim_{k\to\infty}\norm{\bar{y}^{k+1}-y^{k}+\alpha\cG^{-1}d_y^{k-1}}_{\cG}^2= 0$,
and
$\lim_{k\to\infty}\norm{\bar r^{k+1}}^{2}=0$.
Then, by $\cF\succ 0$ and $\cG\succ 0$, we have that
$\{\bar{x}^{k+1}-x^{k}\}\to 0$, $\{\bar{y}^{k+1}-y^{k}\}\to 0$
and $\{\bar{r}^{k}\}\to 0$ as $k\to\infty$.
Also, due to the fact that $\cM\succ 0$ and $\cN\succ 0$, by Proposition \ref{prop:error} we know that
$\{\bar{x}^{k}-x^{k}\}\to 0$ and $\{\bar{y}^{k}-y^{k}\}\to 0$ as $k\to\infty$.
As a result, it holds that
$\{x^{k}-x^{k+1}\}\to0$,
$\{y^{k}-y^{k+1}\}\to 0$,
and $\{r^{k}\}\to 0$ as $k\to\infty$.
Note that the sequence $\{(x^{k+1},y^{k+1},z^{k+1})\}$ is bounded. Thus, it has a convergent subsequence $\{(x^{k_{i}+1},y^{k_{i}+1},z^{k_{i}+1})\}$ which converges to a point, say
$(x^{\infty},y^{\infty},z^{\infty})\in\cX\times\cY\times\cZ$.
We define two nonlinear mappings $F:\cX\times\cY\times\cZ\to\cX$ and $G:\cX\times\cY\times\cZ\to\cZ$ by
$$
F(w):=\partial p(x)+\nabla f(x)+\cA z\quad\mbox{and}\quad G(w):=\partial q(y)+\nabla g(y)+\cB z,\quad \forall (x,y,z)\in\cX\times\cY\times\cZ.
$$
Note that for any $k\ge 0$,
\[
\label{opt:y}
\begin{array}{l}
d_{y}^k-\nabla g(y^k)-\cB\widetilde z^{k+1}+(\hat\Sigma_g+\cT)y^k-y^{k+1}\in\partial q(y^{k+1}).
\end{array}
\]
Since that Assumption \ref{ass}{\rm (ii)} holds, by Clarke's Mean Value Theorem \cite[Proposition 2.6.5]{clarke}, we know that for any $k\ge 1$, there exist two self-adjoint linear operators $0\preceq \cP_{x}^{k}\preceq\hat\Sigma_{f}$ and $0\preceq\cP_{y}^{k}\preceq \hat\Sigma_{g}$ such that
\[
\label{eq:meanvalue-y}
\nabla f(x^{k-1})-\nabla f(x^{k})=\cP_{x}^{k}(x^{k-1}-x^k)
\quad\mbox{and}\quad
\nabla g(y^{k-1})-\nabla g(y^{k})=\cP_{y}^{k}(y^{k-1}-y^k).
\]
Note that \eqref{inclu:subx} holds. Then, by \eqref{opt:y} and \eqref{eq:meanvalue-y} we know that for all $k\ge 1$,
\[
\left\{
\label{opt:2sub}
\begin{array}{r}
d_{x}^{k}-\cP_x^{k+1}(x^k-x^{k+1})
+(\hat\Sigma_{f}+\cS)(x^k-x^{k+1})
+(\tau-1)\sigma\cA r^{k+1}
-\sigma\cA\cB^{*}(y^{k}-y^{k+1})
\qquad
\\[2mm]
\in F(w^{k+1}),
\\[2mm]
d_{y}^k
-\cP^{k+1}_yy^k-y^{k+1}
+(\hat\Sigma_g+\cT)y^k-y^{k+1}
+(\tau-1)\sigma\cB r^{k+1}
\in G(w^{k+1}).
\end{array}
\right.
\]
By taking limits along $\{k_i\}$ as $i\to\infty$ in \eqref{opt:2sub}, we know that
$$
\begin{array}{l}
0\in\partial p(x^\infty)+\nabla f(x^\infty)+\cA z^\infty
\quad\mbox{and}\quad
0\in\partial q(y^\infty)+\nabla g(y^\infty)+\cB z^\infty,
\end{array}
$$
which together with the fact that $\lim_{k\to\infty}r^k= 0$ implies that $(x^\infty, y^\infty, z^\infty)\in\overline\cW$.
Hence, $(x^{\infty},y^{\infty})$ is a solution to problem \eqref{problem:primal} and $z^{\infty}$ is a solution to the dual of problem \eqref{problem:primal}.

By \eqref{ineq:xibound} and Lemma \ref{lemma:sq-sum}, we know that the sequence $\{\|\xi^k\|\}$ is convergent. We can let $\bar w=(x^\infty, y^\infty, z^\infty)$ in all the previous discussions. In this case, $\lim_{k\to \infty} \|\xi^k\| = 0$.
Thus, from the definition of $\{\xi^k\}$, we know that
$$\lim_{k\to\infty}z^{k} =z^\infty,
\quad
\lim_{k\to\infty}y^{k} =y^\infty,
\quad
\mbox{and}
\quad\lim_{k\to\infty} (\hat\Sigma_f+\cS)x^{k} =(\hat\Sigma_f+\cS)x^\infty.
$$
Since that $\lim_{k\to \infty} r^k=0$, it holds that $\{\cA^*x^k\}\to \cA^*x^\infty$ as $k\to\infty$.
Consequently, from the definition of $\cM$ and the fact that $\cM\succ 0$, we can get
$\lim_{k\to\infty}x^{k}=x^{\infty}$, which completes the proof.
\qed
\end{proof}

\end{document}